\documentclass{amsart}
\usepackage{amsmath}
\usepackage{amssymb}
\usepackage{amsthm}
\usepackage{enumerate}
\usepackage[pdftex]{graphicx}
\usepackage{caption}
\theoremstyle{definition}
\newtheorem{definition}{Definition}[section]
\theoremstyle{plain}
\newtheorem{lemma}[definition]{Lemma}
\newtheorem{theorem}[definition]{Theorem}

\newtheorem{proposition}[definition]{Proposition}
\newtheorem{corollary}[definition]{Corollary}
\theoremstyle{remark}
\newtheorem{remark}[definition]{Remark}

\newtheorem{example}[definition]{Example}

\makeatletter
\@namedef{subjclassname@2020}{
  \textup{2020} Mathematics Subject Classification}
\makeatother

\newcommand{\myint}{\operatorname{int}}
\newcommand{\mycl}{\operatorname{cl}}

\DeclareMathSymbol{\mlq}{\mathord}{operators}{``}
\DeclareMathSymbol{\mrq}{\mathord}{operators}{`'}

\begin{document}
\title[Nonvaluational OAGs of finite burden]{Nonvaluational ordered Abelian groups of finite burden}
\author[M. Fujita]{Masato Fujita}
\address{Department of Liberal Arts,
	Japan Coast Guard Academy,
	5-1 Wakaba-cho, Kure, Hiroshima 737-8512, Japan}
\email{fujita.masato.p34@kyoto-u.jp}

\begin{abstract}
	Consider an expansion $\mathcal R=(R,<,+,\ldots)$ of an ordered divisible Abelian group of finite burden defining no nonempty subset $X$ of $R$ which is dense and codense in a definable open subset $U$ of $R$ with $X \subseteq U$.
	We further assume that $\mathcal R$ is nonvaluational, that is, for every nonempty definable subsets $A,B$ of $R$ with $A <B$ and $A \cup B=R$, $\inf\{b-a\;|\;a \in A, b \in B\}=0$.
	Then, $\mathcal R$ is $*$-locally weakly o-minimal.
	We also give a complete description of sets definable in a definably complete expansion of ordered group of burden two if it defines an infinite discrete set.
\end{abstract}


\maketitle

\section{Introduction}\label{sec:intro}

Dolich and Goodrick have extensively studied expansions of ordered Abelian groups of finite burden in \cite{DG,DG2,DG3}. 
We study such structures in this paper.
These structures possibly define unary subsets of open sets $U$ which are dense and codense in $U$.  
In this section, we call such a set \textit{wild} for short.
For instance, $\operatorname(\mathbb R,<,+,\mathbb Z,\mathbb Q)$ is of burden three by \cite[Propositon 3.1]{DG} and defines a wild set $\mathbb Q$.
If the structures under consideration are definably complete, their open core enjoy tame topological property.
In fact, combining \cite[Theorem 1.1]{Fuji_core} with \cite[Corollary 2.13(3)]{DG}, we immediately conclude that every definably complete expansion of an ordered group with strong theory has a locally o-minimal open core.
See \cite{DMS,MS} for the definition of open cores.
This implies that wild topological properties of sets definable in these structures cause from the existence of wild sets.  

We want to extend this result to the non-definably complete case.
A question is whether sets definable in (not necessarily definably complete) expansions of ordered Abelian groups with strong theories enjoy somewhat tame topological properties if the expansions define no wild set.
We give a partial affirmative answer to this question.
We employ a stronger assumption that the structure under consideration is nonvaluational and of finite burden.
In Section \ref{sec:convex}, we show that this structure is a $*$-locally weakly o-minimal structure defined in \cite{Fuji_weak} (cf.\ Theorem \ref{thm:locally_weakly_omin}).
New terms in the above statement are defined in Section \ref{sec:definition}.

The second ingredient of this paper is an extension of the studies in \cite{DG2,DG3} on definably complete expansions of ordered groups of burden two.
In such case, the structure either has an o-minimal open core or defines an infinite discrete closed set.
In Section \ref{sec:burden_two}, we give a complete description of the structure when it defines an infinite discrete closed set.

Throughout, let $\mathcal R=(R,<,+,\ldots)$ be an expansion of an ordered divisible Abelian group.
A structure satisfying weaker conditions is denoted by $\mathcal M=(M,\ldots)$ instead of $\mathcal R$.
The term ‘definable’ means ‘definable in the given structure with parameters’.
We assume that $R$ is equipped with the order topology induced from the linear order $<$ and the topology on $R^n$ is the product topology of the order topology on $R$.
An open interval is a set of the form $(b_1,b_2):=\{x \in R\;|\; b_1<x<b_2\}$ for some $b_1,b_2 \in R \cup \{\pm \infty\}$.
For two subsets $A_1$ and $A_2$ of $R$, $A_1 < A_2$ means that, for every $a_1 \in A_1$ and $a_2 \in A_2$, $a_1<a_2$.
The set $\{a_1+a_2 \in R\;|\; a_1 \in A_1, a_2 \in A_2\}$ is denoted by $A_1+A_2$.
The condition $A_1<\{a_2\}$ is simply denoted by $A_1<a_2$ for $a_2 \in R$.
We define $a_2+A_1$ similarly.
$\myint(A)$ and $\mycl(A)$ denote the interior and the closure of a subset $A$ of $R^n$, respectively.

\section{Definitions}\label{sec:definition}
We quote the definitions of burden, definable Dedekind completion, etc. 
We first recall the definition of burdens (cf.\ \cite{Adler, DG, DG2, DG3,Touchard}).

We fix a complete first-order theory $T$.
Let $p(\overline{x})$ be a partial type.
An \textit{inp-pattern of depth $\kappa$  in $p(\overline{x})$} is a sequence 
$(\phi_{\alpha}(\overline{x};\overline{y})\;|\; \alpha < \kappa)$ of formulas, a sequence $(k_\alpha\;|\; \alpha < \kappa)$ of positive integers and a sequence $(\overline{b}_i^\alpha\;|\; \alpha < \kappa, i < \omega)$ of tuples from some model $\mathcal M$ of $T$ such that:  
\begin{itemize}
	\item $\{\phi_\alpha(\overline{x};\overline{b}_i^\alpha)\;|\; i<\omega\}$ is $k_{\alpha}$-inconsistent for every $\alpha<\kappa$;
	\item $\{\phi_\alpha(\overline{x};\overline{b}_{\eta(\alpha)}^\alpha)\;|\; \alpha<\kappa\}$ is consistent with $p(\overline{x})$ for all map $ \eta:\kappa \to \omega$.
\end{itemize}
The partial type $p(\overline{x})$ has \textit{burden $<\kappa$} if there is no inp-pattern of depth $\kappa$ in $p(\overline{x})$.
If the least $\kappa$ such that the burden of $p(\overline{x})$ is less than $\kappa$ is a successor cardinal with $\kappa=\lambda^+$, then we say that \textit{the burden of $p(\overline{x})$} is $\lambda$.
If the burden of the partial type $x=x$ in a single free variable $x$ exists and is equal to $\kappa$, we say that \textit{the burden of $T$} is $\kappa$.
The theory $T$ is \textit{inp-minimal} if its burden is $1$.
The theory $T$ is \textit{strong} if, for any finite tuple of variables $\overline{x}$, every inp-pattern in the partial type $\overline{x}=\overline{x}$ has finite depth.

Next, we recall the definitions of definable Dedekind completion, local o-minimality and $*$-local weak o-minimality (cf.\ \cite{Fuji_weak, TV}).
Let $\mathcal M=(M,<,\ldots)$ be an expansion of a dense linear order without endpoints.
$\mathcal M$ is \textit{definably complete} if any definable subset $X$ of $M$ has the supremum and  infimum in $M \cup \{\pm \infty\}$.
A \textit{definable gap} is a pair $(A,B)$ of nonempty definable subsets of $M$ such that 
\begin{itemize}
	\item $M=A \cup B$;
	\item $a<b$ for all $a \in A$ and $b \in B$;
	\item $A$ does not have a largest element and $B$ does not have a smallest element.
\end{itemize}
Note that 'being a definable gap' is expressed by a first-order formula when formulas defining $A$ and $B$ are given.
A definably complete structure has no definable gap.

Set $\overline{M}=M \cup \{\text{definable gaps in }M\}$. 
We can naturally extend the order $<$ in $M$ to an order in $\overline{M}$, which is denoted by the same symbol $<$.
For instance, we define $a<(A,B)$ by $a \in A$ when $a \in M$ and $(A,B)$ is a definable gap.
We define $(A_1,B_1) \leq (A_2, B_2)$ by $A_1 \subseteq A_2$ for definable gaps  $(A_1,B_1)$ and $(A_2, B_2)$.
The linearly ordered set $(\overline{M},<)$ is called the \textit{definable Dedekind completion} of $\mathcal M$.

An arbitrary open interval $I$ in $M$ is of the form $(b_1,b_2)$, where $b_1,b_2 \in M \cup \{\pm \infty\}$.
We set $\overline{I}=\{x \in \overline{M}\;|\; b_1<x<b_2\}$.
We define $\overline{I}$ for every interval $I$ in $M$ in the same manner.
Throughout, we use these overlined notations to represent Dedekind completions and their subsets defined above.

In \cite{Fuji_weak}, we do not assume that the addition $+$ is definable.
Consider the case in which $\mathcal M$ is an expansion of an ordered Abelian group and introduce a new concept.
For every definable gap $\alpha=(A,B)$ and $b \in M$, the definable gap $(\{x+b \in M\;|\;x \in A\}, \{x+b \in M\;|\;x \in B\})$ is denoted by $\alpha+b$.
We say that a definable gap $\alpha$ is \textit{nonvaluational (n.v.\ for short)} if $\alpha+b \neq \alpha$ for every $0 \neq b \in M$.
In other words, $\alpha$ is n.v.\  if and only if $\inf\{y-x\;|\; x \in A, y\in B\}=0$.
The term `nonvaluational' was chosen named after nonvaluational weakly o-minimal structures \cite{MMS, Wencel}.
A \textit{valuational} definable gap is a definable gap which is not n.v.
An element $a \in \overline{M}$ is \textit{n.v.}\ if it is either in $M$ or an n.v.\  definable gap.
We call $\mathcal M$ \textit{nonvaluational (n.v.)}\  if there are no valuational definable gaps.
Note that being n.v.\ is a first-order property, that is, for two structures $\mathcal M_1$ and $\mathcal M_2$ which are elementarily equivalent to each other, $\mathcal M_1$ is n.v.\ if and only if $\mathcal M_2$  is n.v.

\begin{example}
	Let $(M,<,+)$ be an ordered subgroup of $(\mathbb R,<,+)$.
	Consider an expansion of $(M,<,+)$.
	Every definable gap $\alpha$ corresponds to a real number $r_{\alpha} \in \mathbb R$.
	Therefore, every definable gap is n.v.
\end{example} 

\begin{example}
	Put $M=\mathbb R \times \mathbb R$.
	Let $<$ be the lexicographic order on $M$ and we define $+:M \times M \ni ((x_1,y_1),(x_2,y_2)) \mapsto (x_1+x_2,y_1+y_2) \in M$.
	Let $\sim$ be the equivalence relation on $M$ defined by $(x_1,y_1) \sim (x_2,y_2) \Leftrightarrow x_1=x_2$.
	Consider the structure $(M,<,+,\sim)$.
	Observe that $A=\{(x,y) \in M\;|\; x \leq 0\}=\{z \in M\;|\; z \sim (0,0) \vee z < (0,0)\}$ and $B=\{(x,y) \in M\;|\; x > 0\}$ are definable sets and $\alpha=(A,B)$ is a valuational definable gap. 
	In fact, $\alpha+(0,y) =\alpha$ for every $y \in \mathbb R$.
\end{example}

For a nonempty definable subset $X$ of $M$, we define $\sup X \in \overline{M} \cup \{+\infty\}$ as follows:
We set $\sup X=+\infty$ when, for any $a \in M$, there exists $x \in X$ with $x>a$.
Assume that there exists $z \in M$ such that $x<z$ for every $x \in X$.
Set $B=\{y \in M\;|\; \forall x \in X\ y>x\}$ and $A=M \setminus B$.
If $B$ has a smallest element $m$, we set $\sup X=m$.
If $A$ has a largest element $m'$, we set $\sup X=m'$.
Finally, if $(A,B)$ is a definable gap, we set $\sup X=(A,B) \in \overline{M}$.

$\mathcal M$ is a \textit{locally o-minimal structure} if, for every definable subset $X$ of $M$ and for every point $a \in M$, there exists an open interval $I$  such that $a \in I$ and $X \cap I$ is a union of a finite set and finitely many open intervals \cite{TV}.
$\mathcal M$ is a \textit{$*$-locally weakly o-minimal structure} if, for every definable subset $X$ of $M$ and for every point ${a}\in \overline{M}$, there exists an open interval $I$  such that ${a} \in \overline{I}$ and $X \cap I$ is a union of a finite set and finitely many open convex sets.
Note that a $*$-locally weakly o-minimal structure is locally o-minimal by \cite[Proposition 2.4]{Fuji_weak}. 

\section{Finite burden case}\label{sec:convex} 
We consider the case in which $\mathcal R$ is of finite burden.
Suppose that $\mathcal R$ is $\kappa$-saturated for a sufficiently large cardinal $\kappa$ in this section.
We prepare lemmas to prove Theorem \ref{thm:locally_weakly_omin}.
For every $0<b \in R$ and an n.v.\  element $a \in \overline{R}$, the set $\{x \in R\;|\; a<x<a+b\}$ is denoted by $(a,a+b)$.
We define $(a-b,a)$ similarly.

We prove two modified versions of \cite[Lemma 2.10]{DG}.
The following lemma is the first one:
\begin{lemma}\label{lem:DM_minor_ext}
Let $\alpha \in \overline{R}$ be an n.v.\  element.
Let $D_i$ be infinite definable discrete sets and $\varepsilon_i>0$ for $i \in \mathbb N$ satisfying the following conditions:
\begin{enumerate}
	\item[(1)] $D_{i+1} \subseteq (\alpha,\alpha +\varepsilon_i/3)$;
	\item[(2)] If $x \in D_i$, then $(x-\varepsilon_i,x+\varepsilon_i) \cap D_i = \{x\}$;
\end{enumerate} 
Then $\operatorname{Th}(\mathcal R)$ is not strong.
\end{lemma}
\begin{proof}
	First, if $\alpha \in R$, then this lemma is identical to \cite[Lemma 2.10]{DG}.
	Therefore, we may assume that $\alpha$ is a definable gap $(A,B)$.
	For $i \in \mathbb N$, we construct infinite definable discrete sets $E_i$ and $t_i \in R$ such that $E_i \subseteq (0,\varepsilon_i/3)$ and $t_i+E_i \subseteq D_i$.
		
	For $j,N \in \mathbb N$, consider the formula $$\phi_{j,N}(r):= \mlq |(-\infty,r) \cap D_j| \geq N \wedge |(r,\infty) \cap D_j| \geq N\mrq,$$ where the cardinality of a set $S$ is denoted by $|S|$.
	The family of formulas $\{\phi_{j,N}\}_{N \in \mathbb N}$ is finitely satisfiable because $D_j$ is an infinite set.
	Since $\mathcal R$ is $\kappa$-saturated, we can find $t_j \in R$ with $\mathcal R \models \bigwedge_{N \in \mathbb N}\phi_{j,N}(t_j)$.
	Put $E_j:=\{x \in R\;|\; x>0,\ x+t_j \in D_j\}$.
	The inclusion $t_j+E_j \subseteq D_j$ is obvious.
	
	Let $e \in E_j$.
	It is obvious $e>0$.
	By clause (1), we have $t_j>\alpha$ and $e+t_j-\alpha<\varepsilon_{j-1}/3$.
	Therefore, $e<\varepsilon_{j-1}/3$. 
	We have showed $E_j \subseteq (0,\varepsilon_j/3)$.
	
	Now, apply 	\cite[Lemma 2.10]{DG} to $E_j$'s, then we obtain this lemma.	
\end{proof}

We get the following corollary:
\begin{corollary}\label{cor:accumulation}
	Suppose $\operatorname{Th}(\mathcal R)$ is strong.
	Let $D$ be an infinite discrete definable subset of $R$.
	Then, for every n.v.\  $a \in \overline{R}$, there exists an open interval $I$ such that $a \in \overline{I}$ and $D \cap I$ is a finite set.
\end{corollary}
\begin{proof}
	If $a \in R$, this corollary is \cite[Corollary 2.13(1)]{DG}.
	Suppose $a=(A,B)$ is a definable gap.
	We may assume $a<D$ without loss of generality.
	Assume for contradiction that $D \cap I$  is an infinite set for arbitrary open interval $I$ with $a \in \overline{I}$.
	Put $D(r):=\{x \in D\;|\; (x-r,x+r) \cap D =\{x\}\}$.
	Let $\psi_N(r)=\mlq r>0 \wedge |D(r)| \geq N \mrq$ for $N \in \mathbb N$.
	We can find $r_N$ with $\mathcal R \models \psi_N(r_N)$ because $D$ is discrete and infinite.
	By saturation, there exists $\varepsilon_0>0$ such that $\mathcal R \models \bigwedge_{N \in \mathbb N}\psi_N(\varepsilon_0)$.
	Put $D_0:=D(\varepsilon_0)$. 
	Since $a$ is n.v., we can find an open interval $I$ with $a \in \overline{I}$ and $I<a+\varepsilon_0/3$.
	Considering $D \cap I$ instead of $D$, we may assume that $D \subseteq (a,a+\varepsilon_0/3)$.
	We inductively define $D_1, D_2$ in the same manner as above.
	We have constructed a family of infinite discrete definable subsets $D_i$ of $R$ satisfying the conditions in Lemma \ref{lem:DM_minor_ext}.
	This contradicts the assumption that $\operatorname{Th}(\mathcal R)$ is strong.
\end{proof}

Let $X$ be a definable subset of $R$.
For every $x \in X$, the \textit{maximal convex subset $D(X,x)$ of $X$ containing $x$} is a definable set defined by the following formula:
\begin{align*}
\Phi_{X,x}(y)&=\mlq (y=x) \vee (y<x \wedge \forall z \ y \leq z<x \rightarrow z \in X) \\
&\vee (x<y \wedge \forall z \ x<z \leq y \rightarrow z \in X)\mrq.
\end{align*}
Observe that, if $x' \in D(X,x)$, $D(X,x')=D(X,x)$.
We put $l(X,x):=\sup\{r>0\;|\; \exists y_1,y_2 \in D(X,x)\ |y_1-y_2|=r\} \in \overline{R} \cup \{\infty\}$.
Let $\#_{\text{convex}}S$ be the cardinality of the family of maximal convex subsets of $S$ for every definable subset $S$ of $R$.
We use these notations in the rest of this section.
Observe that $D(X,x)=\{x\}$ and $l(X,x)=0$ if $X$ is discrete.
Note that the conditions that $\#_{\text{convex}}X \geq N$ and $l(X,x) >r$ are first-order for $N \in \mathbb N$ and $0<r \in R$.
For instance, the former is expressed by
$$\exists x_1,x_2 \in X\ D(X,x_1) \cap D(X,x_2)=\emptyset$$
if $N=2$,  and the latter is described as
$$\exists \varepsilon>0\ \exists y \in D(X,x)\ (y-r-\varepsilon,y+r+\varepsilon) \subseteq D(X,x).$$ 

The following lemma is the second modified version of \cite[Lemma 2.10]{DG}:
\begin{lemma}\label{lem:key_extended}
Suppose $\mathcal R$ is n.v.
Let $N \geq 2$ be a natural number.
Let $\{X_n\}_{n=1}^N$ be a family of definable subsets of $R$ having infinitely many maximal convex subsets.
Let $\{\varepsilon_n\}_{n=1}^N$ be a decreasing family of positive elements in $R$.
Suppose the following condition is satisfied:
\begin{enumerate}
	\item[(1)] $X_{n+1} \subseteq (0,\varepsilon_n/3)$,
	\item[(2)] If $x \in X_n$, then $(x-2\varepsilon_n,x+2\varepsilon_n) \cap X_n \subseteq D(X_n,x)$.
\end{enumerate}
Then $\operatorname{Th}(\mathcal R)$ is of burden $\geq N-1$.
\end{lemma}
\begin{proof}
	Let $D$ be a maximal convex subset of $X_n$.
	Put $\alpha = \sup D \in \overline{R}$.
	Since $\alpha$ is n.v., $\{x \in D\;|\; \exists y \ x<y<x+\varepsilon_N/(4N),\ y \notin X_n\}=D \cap (\alpha-\varepsilon_N/(4N),\alpha)$ is a nonempty convex set.
	Therefore, by considering $\{x \in X_n\;|\; \exists y \ x<y<x+\varepsilon_N/(4N),\ y \notin X_n\}$ instead of $X_n$, we may assume that $l(X_n,x)<\varepsilon_N/(2N)$ for all $x \in X_n$.
	
	For $2 \leq n \leq N$, we will pick pairwise disjoint open intervals $I_i^n$ for $i \in \mathbb N$ satisfying the following conditions:
	\begin{enumerate}
		\item[(i)] $I_0^n < \varepsilon_{n-1}/3$;
		\item[(ii)] $I_{k+1}^n<I_k^n$ for $k \in \mathbb N$;
		\item[(iii)] $x-y>\varepsilon_n$ for every $k \in \mathbb N$, $x \in I_k^n$ and $y \in I_{k+1}^n$;
		\item[(iv)] $I_i^n \cap X_n$ consists of infinitely many maximal convex subsets.
	\end{enumerate}
	Put $t_0=\varepsilon_{n-1}/3$ and $X_{n,0}:=X_n$.
	For every $k \in \mathbb N$, let $\psi_k^n(r)$ be the formula expressing that both $X_{n,0} \cap \{x<r\}$ and $X_{n,0} \cap \{x>r\}$ have at least $k$ distinct maximal convex subsets. 
	Observe that $\mathcal R \models \exists r \ \psi_k^n(r)$ for $k \in \mathbb N$ using the assumption that $X_{n,0}$ has infinitely many maximal convex subsets.
	There exists $r_0>0$ such that $\mathcal R \models \bigwedge_{k \in \mathbb N} \psi_k^n(r_0)$ because $\mathcal R$ is $\kappa$-saturated.
	$X_{n,0} \cap \{x<r_0\}$ and $X_{n,0} \cap \{x>r_0\}$ have infinitely many maximal convex subsets. 
	Set $I_0^n:=(r_0,t_0)$, $t_1:=r_0-3\varepsilon_n$ and $X_{n,1}:=X_{n,0} \cap (-\infty,t_1)$.
	Observe that $X_{n,1}$ has infinitely many distinct maximal convex subsets by clause (2) and the choice of $r_0$.
	Repeating in this manner inductively, we construct  pairwise disjoint open intervals $I_k^n$ or $k \in \mathbb N$  satisfying conditions (i) through (iv).
	Let $l_k^n$ and $r_k^n$ be the left and right endpoints of $I_k^n$, respectively.
	
	Let $Y_1:=X_1$ and $$Y_n:=\{y+d\;|\; y \in Y_{n-1},\ d \in X_n\}$$ for $n >1$.
	Clauses (1) and (2) imply $\varepsilon_{n+1}<\varepsilon_n/2$.
	We show the following claims simultaneously by induction on $n$.
	\begin{enumerate}
		\item[(a)] For $n>1$, $d \in X_n$ and $y \in Y_{n-1}$, $D(Y_n,y+d)=D(Y_{n-1},y)+D(X_n,d)$;
		\item[(b)] $l(Y_n,x)<n\varepsilon_N/N$ and 
		\item[(c)] $(x-\varepsilon_n,x+\varepsilon_n) \cap Y_n \subseteq D(Y_n,x)$ for $n>0$ and $x \in Y_n$.
	\end{enumerate}
	Clauses (b) and (c) are obvious if $n=1$.
	Assume $n>1$.
	It is obvious that $D(Y_{n-1},y)+D(X_n,d)$ is a convex set for $d \in X_n$ and $y \in Y_{n-1}$.
	We have $l(D(Y_{n-1},y)+D(X_n,d),y+d)<n\varepsilon_N/N$ because $l(Y_{n-1},y)<(n-1)\varepsilon_N/N$ and $l(X_n,d)<\varepsilon_N/N$.
	Suppose that $(D(Y_{n-1},y)+D(X_n,d)) \cap (D(Y_{n-1},y')+D(X_n,d')) \neq \emptyset$ for some $y,y' \in Y_{n-1}$ and $d,d' \in X_n$.
	There exist $z \in D(Y_{n-1},y)$, $z' \in D(Y_{n-1},y')$, $e \in D(X_n,d)$ and $e' \in D(X_n,d')$ such that $z+e=z'+e'$.
	Since $l(X_n,x)<\varepsilon_N/(2N)$ for every $x \in X_n$, we have $z = z' + (e'-e) \in (z'-\varepsilon_N/N,z'+\varepsilon_N/N) \cap Y_{n-1} \subseteq D(Y_{n-1},z)$ by the induction hypothesis.
	Therefore, $D(Y_{n-1},y)=D(Y_{n-1},z)=D(Y_{n-1},z')=D(Y_{n-1},y')$.
	We have $|e-e'| =|z-z'|<2(n-1)\varepsilon_N/N$ by the above equality and clause (b).
	Clause (2) implies that $D(X_n,d)=D(X_n,d')$.
	We have shown clause (a).
	Clause (b) is obvious from clause (a) and the inequality $l(X_n,d)<\varepsilon_N/N$.
	
	Finally, we prove clause (c).
	Let $x' \in (x-\varepsilon_n,x+\varepsilon_n) \cap Y_n$.
	Choose $y,y' \in Y_{n-1}$ and $d,d' \in X_n$ so that $x=y+d$ and $x'=y'+d'$.
	We have $|y-y'| \leq |x-x'|+|d'-d|<\varepsilon_n+\varepsilon_{n-1}/3<5\varepsilon_{n-1}/6$ by clause (1).
	By the induction hypothesis, we have $D(Y_{n-1},y)=D(Y_{n-1},y')$.
	Next we have $|d-d'| \leq |x-x'|+|y-y'|<\varepsilon_n+n\varepsilon_N/N<2\varepsilon_n$.
	By clause (2), we have $D(X_n,d)=D(X_n,d')$.
	We get $x' \in D(Y_{n-1},y')+D(X_n,d')=D(Y_{n-1},y)+D(X_n,d)=D(Y_n,x)$.
		
	Now, for $1 \leq n < N$, let $\phi_n(x;y,z)$ be the following formula:
	\begin{align*}
		\phi_n(x;y,z) &:=\mlq \exists w (w \in Y_n \wedge w< x \wedge y<x-w<z) \mrq.
	\end{align*}
	We show that $\phi_n(x;l_i^{n+1},r_i^{n+1}) \wedge \phi_n(x;l_k^{n+1},r_k^{n+1})$ is inconsistent if $i \neq k$.
	Suppose for contradiction that there exists $a \in R$ with $\mathcal R \models \phi_n(a;l_i^{n+1},r_i^{n+1}) \wedge \phi_n(a;l_k^{n+1},r_k^{n+1})$.
	We find $y_i,y_k \in Y_n$ such that $y_i<a$, $y_k<a$, $l_i^{n+1}<a-y_i<r_i^{n+1}$ and $l_k^{n+1}<a-y_k<r_k^{n+1}$.
	If $D(Y_n,y_i)=D(Y_n,y_k)$, we have $a-y_i \in I_i^{n+1}$, $a-y_k \in I_k^{n+1}$ and $|(a-y_i)-(a-y_k)|=|y_i-y_k|< n\varepsilon_N/N<\varepsilon_n$.
	This contradicts clauses (ii) and (iii).
	If $D(Y_n,y_i) \neq D(Y_n,y_i)$, we may assume that $y_i<y_k$ without loss of generality.
	We have $a-y_i=a-y_k+(y_k-y_i) >\varepsilon_n$ by clause (c).
	We have $a-y_i \notin I_i^{n+1}$ because $I_i^{n+1} \subseteq (0,\varepsilon_n/3)$ by clauses (i) and (ii).
	This contradicts $ l_i^{n+1}<a-y_i<r_i^{n+1}$.
	
	Now let $\eta:\{1,2,\ldots,N-1\} \to \mathbb N$.
	We show the partial type $$\Gamma (x) :=\bigwedge_{n=1}^{N-1} \phi_n(x;l_{\eta(n)}^{n+1},r_{\eta(n)}^{n+1})$$ is consistent.
	This deduces that $\operatorname{Th}(\mathcal R)$ is of burden $\geq N-1$.
	Choose $y_1 \in Y_1=X_1$.
	Put $J_1=(y_1+l_{\eta(1)}^2, y_1+r_{\eta(1)}^2)$.
	Choose $d_1 \in X_2 \cap I_{\eta(1)}^2$ so that $D(X_2 \cap I_{\eta(1)}^2, d_1)$ is neither the leftmost and rightmost maximal convex subsets of $X_2 \cap I_{\eta(1)}^2$. 
	It is possible by clause (iv).
	Put $y_2=y_1+d_1$ and $J_2:=(y_2+l_{\eta(2)}^3, y_2+r_{\eta(2)}^3)$.
	We can find $e_{l1},e_{r1} \in X_2 \cap I_{\eta(1)}^2$ such that $e_{l1}<d_1<e_{r1}$, $D(X_2 \cap I_{\eta(1)}^2, e_{l1}) \neq D(X_2 \cap I_{\eta(1)}^2, d_1)$ and $D(X_2 \cap I_{\eta(1)}^2, e_{r1}) \neq D(X_2 \cap I_{\eta(1)}^2, d_1)$.
	We have $(y_1+e_{l1},y_1+e_{r1}) \subseteq J_1$.
	Observe that $l_{\eta(2)}^3<\varepsilon_2/3$ and $r_{\eta(2)}^3 \leq \varepsilon_2/3$ by clauses (i) and (ii).
	Clause (2) implies that $d_1-e_{l1} \geq 2\varepsilon_2$ and $e_{r1}-d_1 \geq 2\varepsilon_2$.
	These imply that $J_2 \subseteq (y_1+e_{l1},y_1+e_{r1})$.
	We have proved $J_2 \subseteq J_1$.
	
	We construct $J_3, \ldots, J_{N-1}$ in the same manner.
	It is easy to check that every element in $J_n$ satisfies $\phi_n(x;l_{\eta(n)}^{n+1},r_{\eta(n)}^{n+1})$.
	Every element in $J_{N-1}$ satisfies $\Gamma(x)$.
\end{proof}

\begin{lemma}\label{lem:key_extended2}
	Suppose $\mathcal R$ is n.v.
	Let $N \geq 2$ be a natural number and $\alpha \in \overline{R}$.
	Let $\{X_n\}_{n=1}^N$ be a family of definable subsets of $R$ having infinitely many maximal convex subsets.
	Let $\{\varepsilon_n\}_{n=1}^N$ be a decreasing family of positive elements in $R$.
	Suppose the following condition is satisfied:
	\begin{enumerate}
		\item[(1)] $X_{n+1} \subseteq (\alpha,\alpha+\varepsilon_n/3)$,
		\item[(2)] If $x \in X_n$, then $(x-2\varepsilon_n,x+2\varepsilon_n) \cap X_n \subseteq D(X_n,x)$.
	\end{enumerate}
	Then $\operatorname{Th}(\mathcal R)$ is of burden $\geq N-1$.
\end{lemma}
\begin{proof}
	The proof of this lemma is almost the same as that of Lemma \ref{lem:DM_minor_ext}.
	We use Lemma \ref{lem:key_extended} instead of \cite[Lemma 2.10]{DG} and count the number of maximal convex subsets in place of counting the number of points.
%
%
%
%
\end{proof}

\begin{corollary}\label{cor:accumulation2}
	Suppose $\mathcal R$ is n.v.\ and $\operatorname{Th}(\mathcal R)$ is of finite burden.
	Let $U$ be a definable open subset of $R$ having infinitely many maximal convex subsets.
	Then, for every $a \in \overline{R}$, there exists an open interval $I$ with $a \in \overline{I}$ such that $U \cap I$ is a union of finitely many open convex set.
\end{corollary}
\begin{proof}
	Assume for contradiction that $U \cap I$ has infinitely many maximal convex subsets for arbitrary open interval $I$ with $a \in \overline{I}$.
	Put $U(r):=\{x \in U\;|\; l(U,x) > 4r\}$ for $r>0$.
	Let $\psi_N(r)=\mlq r>0 \wedge \#_{\text{convex}}U(r) \geq N \mrq$ for $N \in \mathbb N$.
	We can easily verify $\mathcal R \models \exists r\ \psi_N(r)$.
	By saturation, there exists $\varepsilon_1>0$ such that $\mathcal R \models \bigwedge_{N \in \mathbb N}\psi_N(\varepsilon_1)$.
	Put $U_1:=U(\varepsilon_1)$. 
	$U_1$ has infinitely many maximal convex subsets.
	Set $V_1:=\{x \in U_1\;|\; (x-2\varepsilon_1,x+2\varepsilon_1) \subseteq U_1\}$.
	By the definition of $U_1$, for every maximal convex subset $C$ of $U_1$, $C \cap V_1$ is nonempty and convex.
	This $V_1$ satisfies clause (2) of Lemma \ref{lem:key_extended2}.
	Since $\alpha$ is n.v, we can find an open interval $I$ with $\alpha \in \overline{I}$ and $I<\alpha+\varepsilon_1/3$.
	Considering $U \cap I$ instead of $U$, we may assume that $U \subseteq (\alpha,\alpha+\varepsilon_1/3)$.
	We inductively define $V_2, V_3, \ldots$ in the same manner as above.
	We have constructed a family of infinitely many definable subsets $V_i$ of $R$ satisfying the conditions in Lemma \ref{lem:key_extended2}.
	This contradicts the assumption that $\operatorname{Th}(\mathcal R)$ is of finite burden.
\end{proof}

\begin{corollary}\label{cor:Cantor-like}
Let $\mathcal R$ be as in Corollary \ref{cor:accumulation2}.
$\mathcal R$ does not define a nowhere dense subset of $R$ having no isolated points.
\end{corollary}
\begin{proof}
	Suppose for contradiction that there exists a definable nowhere dense subset $X$ of $R$ having no isolated points.
	Take $a \in X$.
	Observe that $a$ is not isolated in the closure $\mycl(X)$.
	For every open interval $I$ containing $a$, $I \setminus \mycl(X)$ is a definable open subset having infinitely many maximal convex subsets.
	This contradicts to Corollary \ref{cor:accumulation2}.
\end{proof}

We are now ready to prove our main theorem.
\begin{theorem}\label{thm:locally_weakly_omin}
	Suppose $\mathcal R$ is nonvaluational and $\operatorname{Th}(\mathcal R)$ is of finite burden.
	We further assume that $\mathcal R$ defines no nonempty subset $X$ of $R$ which is dense and codense in a definable open subset $U$ of $R$ with $X \subseteq U$.
	Then, $\mathcal R$ is $*$-locally weakly o-minimal.
\end{theorem}
\begin{proof}
	Let $a \in \overline{R}$ and $X$ be a definable subset of $R$.
	We have only to prove that there exists an open interval $I$ such that $a \in \overline{I}$ and $I \cap X$ is a union of finitely many convex sets.
	Set $U=\myint(X)$.
	By Corollary \ref{cor:accumulation2}, we find an open interval $I_1$ such that $a \in \overline{I_1}$ and $U \cap I_1$ is a union of finitely many open convex sets.
	Put $Y=X \setminus U$.
	Then $Y$ is partitioned as $Y=Y_1 \cup Y_2 \cup Y_3$ satisfying the following conditions:
	\begin{itemize}
		\item $Y_1$, $Y_2$ and $Y_3$ are definable;
		\item $Y_1$ is either empty or has a definable open subset $V$ of $R$ such that $Y_1 \subseteq V$ and $Y_1$ is dense and codense in $V$;
		\item $Y_2$ is either empty or a nowhere dense subset of $R$ having no isolated points;
		\item $Y_3$ is either empty or discrete.
	\end{itemize}
	$Y_1=\emptyset$ by the assumption of this theorem and $Y_2=\emptyset$ by Corollary \ref{cor:Cantor-like}.
	Therefore, $Y$ is discrete and, by Corollary \ref{cor:accumulation}, there exists an open interval $I_2$ such that $a \in \overline{I_2}$ and $Y \cap I_2$ is a finite set.
	Put $I=I_1 \cap I_2$.
	We have $a \in \overline{I}$ and $I \cap X$ is a union of finitely many convex sets.
\end{proof}

\begin{remark}
	Being n.v.\ is a first-order property.
	By \cite[Proposition 2.16]{Fuji_weak}, the  conclusion of Theorem \ref{thm:locally_weakly_omin} is preserved by elementary equivalence.
	Therefore, Theorem \ref{thm:locally_weakly_omin} holds without the assumption of saturation of $\mathcal R$. 
\end{remark}

\begin{example}
	We cannot drop the assumption that $\mathcal R$ is n.v.\  in Theorem \ref{thm:locally_weakly_omin}.
	A \textit{standard simple product} of two structures is defined in \cite[Definition 2.5]{Fuji_product} and \cite[Definition 3.52]{DG3}.
	We do not repeat the definition of standard simple products here.
	Let $\mathcal M_1=\mathcal M_2=(\mathbb R,<,+,0,\mathbb Z)$ and $\mathcal R$ be the standard simple product of $\mathcal M_1$ and $\mathcal M_2$.
	By \cite{DG}, $(\mathbb R,<,+,0,\mathbb Z)$ has dp-rank $2$.
	$\mathcal R$ is of finite dp-rank by \cite[Proposition 3.53]{DG3}.
	Recall that dp-rank of an NIP theory $T$ (cf.\ \cite{Simon_NIP}) is equal to its burden by \cite[Proposition 1.10]{Adler}.
	The order in $R=\mathbb R \times \mathbb R$ is the lexicographic order in $\mathbb R \times \mathbb R$. 
	$A:=\{(x,y) \in \mathbb R \times \mathbb R=R\;|\; x \leq 0\}$ and $\mathbb R \times \mathbb Z$ are definable in $\mathcal R$.
	The pair $(A, R \setminus A)$ defines a valuational definable gap $\alpha$.
	For every open interval $I$ in $R$ with $\alpha \in \overline{I}$, $I \cap (\mathbb R \times \mathbb Z)$ is an infinite discrete set.
	This example violates Theorem \ref{thm:locally_weakly_omin}.
\end{example}

\section{Burden two case}\label{sec:burden_two}

In this section, we consider the case in which $\mathcal R$ is definably complete and of burden two.
Recall that an expansion of a dense linear order possibly with endpoints is \textit{o-minimal} if every unary definable set is a union of a finite set and finitely many open intervals \cite{vdD}. 

\begin{theorem}\label{thm:local_omin_case}
	Suppose $\mathcal R$ is definably complete.
	$\mathcal R$ is of burden two and defines an infinite discrete subset of $R$ if and only if there exist $0<c \in R$, a definable infinite discrete subgroup $G$ of $R$, a definably complete inp-minimal expansion $\mathcal G$ of $(G,<,+)$ and an o-minimal structure $\mathcal I$ whose universe is $I:=[0,c)$ such that 
	\begin{itemize}
		\item $c \in G$ and $g \geq c$ for every $0<g \in G$;
		\item $(R,<,+,G) \equiv (\mathbb R,<,+,\mathbb Z)$;
		\item $\mathcal R$ is interdefinable with the standard simple product of $\mathcal G$ and $\mathcal I$ under the identification $G \times I \ni (g,t) \mapsto g+t \in R$. 
	\end{itemize}
\end{theorem}
\begin{proof}
	The `only if' part of this theorem follows from \cite[Proposition 1.24]{Touchard} because o-minimal structures are dp-minimal.
	
	We prove the `if' part.	
	By \cite[Corollary 2.10]{DG2} and \cite[Lemma 2.3]{Fuji3}, $\mathcal R$ is locally o-minimal.
	By \cite[Theorem 1.3]{DG3}, we can also choose a definable discrete closed subgroup $G$ of $R$ such that $(R,<,+,G) \equiv (\mathbb R,<,+,\mathbb Z)$ and every discrete subset of $R$ definable in $\mathcal R$ is definable in $(R,<,+,G)$.
	
	Let $\mathcal R_{\text{lin}}:=(R,<,+,G)$.
	Set $c:=\inf\{g \in G\;|\; g>0\}$.
	Observe that $c \in G$ because $G$ is closed in $R$ by \cite[Proposition 2.8(1)]{FKK}.
	Put $I:=[0,c)$.
	Let $\mathcal I$ be the structure $I_{\mathcal R\text{-def}}$ on $I$ generated by every subset of $I^n$, $n>0$, definable in $\mathcal R$ (cf.\ \cite[Definition 2]{KTTT}).
	We define the structure $\mathcal G:=G_{\mathcal R\text{-def}}$ on $G$ in the same manner.
	We can easily show that $\mathcal G$ is definably complete because $G$ is closed in $R$.
	First we prove $\mathcal I$ is o-minimal.
	\medskip
	
	\textbf{Claim 1.} For every definable subset $X$ of $R$, $X \cap I$ is a union of a finite set and a finitely many open intervals.
	\begin{proof}[Proof of Claim 1]
		We may assume $X \subseteq I$ by considering $X \cap I$ instead of $X$.
		We can reduce to the case in which $X$ has a nonempty interior in the same manner as the proof of \cite[Corollary 4.4]{Fuji4}.
		We omit the details.
		We may assume that $X$ is discrete and closed by \cite[Lemma 2.3]{Fuji3}.
		
		Let $\mathcal L:=(<,+,G)$ be the language of $\mathcal R_{\text{lin}} $.
		$X$ is definable in $\mathcal R_{\text{lin}} $ because $X$ is discrete.
		Let $\phi(x;\overline{a})$ be an $\mathcal L$-formula with the parameters $\overline{a}$ defining the set $X$.
		Let $m$ be the arity of $\overline{a}$.
		By considering $$\forall t >0\  (t \in G \rightarrow 0 \leq x < t) \wedge \phi(x,\overline{a})$$	instead of $\phi(x,\overline{a})$, we may assume that $\{x \in R\;|\; \mathcal R_{\text{lin}} \models \phi(x,\overline{b})\}$ is contained in $I$ for every $\overline{b} \in R^m$.
		By \cite[Theorem 25]{KTTT} and the elementary equivalence $\mathcal R_{\text{lin}} \equiv (\mathbb R,<,+,\mathbb Z)$, there exist $M, N \in \mathbb N$ such that, for every $\overline{b} \in R^m$, $\{x \in R\;|\; \mathcal R_{\text{lin}} \models \phi(x,\overline{b})\}$ is a union of at most $M$  open intervals and at most $N$ points.
		In particular, $X$ is a finite set.
	\end{proof}
	
	Claim 1 implies that $\mathcal I$ is o-minimal.
	
	Next we prove that $\mathcal R$ is interdefinable with the standard simple product of $\mathcal G$ and $\mathcal I$.
	Let $(X(k))_{k\in G^n}$ be a uniform $\mathcal R$-definable family of subsets of $I^m$, that is, $\bigcup_{k \in G^n} \{k\} \times X(k)$ is definable in $\mathcal R$.
	We say that $(X(k))_{k \in G^n}$ is \textit{of finite type} if there are finitely many $\mathcal I$-definable subsets $Y_1,\ldots, Y_l$ of $I^m$ such that, for every $k \in G^n$, $X(k)=Y_i$ for some $1 \leq i \leq l$.
	\medskip
	
	\textbf{Claim 2.} Let $(X(k))_{k\in G^n}$ be a uniform $\mathcal R$-definable family of subsets of $I^m$.
	If $X(k)$ is a finite set for every $k \in G^n$, then $\bigcup_{k \in G^n}X(k)$ is a finite set.
	\begin{proof}[Proof of Claim 2]
		This claim follows from \cite[Proposition 2.8(1,6,11)]{FKK} and the fact $\mathcal I$ is o-minimal.
		We omit the details.
	\end{proof}

	\textbf{Claim 3.} Every uniform $\mathcal R$-definable family of subsets of $I^m$ is of finite type.
	\begin{proof}[Proof of Claim 3]
		Let $(X(k))_{k\in G^n}$ be a uniform $\mathcal R$-definable family of subsets of $I^m$.
		We show that $(X(k))_{k\in G^n}$ is of finite type.
		Let $\pi:R^m \to R^{m-1}$ be the coordinate projection forgetting the last coordinate.
		We consider that $R^0$ is a singleton with the trivial topology. 
		For every $x \in R^{m-1}$ and a definable subset $S$ of $R^m$, set $S_x:=\{y \in R\;|\;(x,y) \in S\}$.
		First we reduce to the case in which $\myint(X(k)_x)=\emptyset$ for every $x \in R^{m-1}$.
		
		For $k \in G^n$, set $J (k):=\bigcup_{x \in \pi(X(k))} \{x\} \times (\myint(X(k)_x))$ and $P(k)=X(k) \setminus J(k)$.
		Observe that $J(k)_x$ is open and $\myint(P(k)_x)=\emptyset$ for every $x \in I^{m-1}$. 
		$(P(k))_{k \in G^n}$ is of finite type by the hypothesis of reduction.
		If $(J(k))_{k \in G^n}$ is also of finite type, $(X(k))_{k \in G^n}$ is trivially of finite type.
		We have to prove that $(J(k))_{k \in G^n}$ is of finite type.
		To prove this, we reduce to simpler cases in a step-by-step manner.
		
		For $k \in G^n$, consider the set $Y(k):=\bigcup_{x \in I^{m-1}}\{x\} \times  (\mycl_I(J(k)_x) \setminus J(k)_x)$.
		$(Y(k))_{k \in G^n}$ is of finite type by the hypothesis of reduction.
		This fact and uniform finiteness property of o-minimal structures (cf.\ \cite[Chapter 3, Lemma 2.13]{vdD}) imply that there exists a positive integer $L$ such that, for every $x \in I^{m-1}$ and $k \in G^n$, the cardinality $l(x,k)$ of the $\mathcal I$-definable set $Y(k)_x$ is at most $L$.
		We want to reduce to the case where $l(x,k)$ is independent of choice of $x$ and $k$ whenever $Y(k)_x \neq \emptyset$.
		
		For every $1 \leq l \leq L$ and $k \in G^n$, put $Z\langle l \rangle(k):=\pi^{-1}(\{x \in I^{m-1}\;|\; l(x,k)=l\}) \cap Y(k)$.
		Observe that $Y(k)=\bigcup_{l=1}^L Z\langle l \rangle(k)$ for $k \in G^n$.
		Put $J\langle l \rangle(k):=J(k) \cap \pi^{-1}(Z\langle l \rangle(k))$.
		We have $J(k)=\bigcup_{l=1}^LJ\langle l \rangle(k)$.
		If $(J\langle l \rangle(k))_{k \in G^n}$ is of finite type for each $1 \leq l \leq L$,  $(J(k))_{k \in G^n}$ is so.
		We may assume that $l(x,k)$ is independent of choice of $x$ and $k$ whenever $Y(k)_x \neq \emptyset$, say $l(x,k)=l$, by considering $(J\langle l \rangle(k))_{k \in G^n}$ instead of $(J(k))_{k \in G^n}$.
		
		For every $k \in G^n$ and $x \in \pi(Y(k))$, let $y(x,k,j)$ be the $j$-th smallest element of $Y(k)_x$ for $1 \leq j \leq l$ and put $y(x,k,l+1)=c$ for convenience.
		Let $D$ be the set of sequences of length $l$ whose entries are zeros and ones.
		For $d:=(d_1,\ldots, d_l) \in D$ and $k \in G^n$, consider the sets
		\begin{align*}
			&E \langle d \rangle(k):=\{x \in I^{m-1}\;|\; \text{the interval } (y(x,k,p),y(x,k,p+1)) \text{ is }\\
			&\qquad \text{contained in } J(k)_x \text{ iff } d_p=1 \text{ for each } 1 \leq p \leq l\},\\
			&J [d ](k):=J(k)\cap \pi^{-1}(E\langle d \rangle(k)) \text{ and }\\
			&Y[ d ](k):=Y(k) \cap \pi^{-1}(E\langle d \rangle(k)).
		\end{align*}
		We obviously have $J(k)=\bigcup_{d \in D}J[d ](k)$ and $Y(k)=\bigcup_{d \in D}Y[d ](k)$ for each $k \in G^n$.
		By the assumption of reduction, $(Y[d](k))_{k \in G^n}$ is of finite type for every $d \in D$.
		By the definitions of $(J[d](k))_{k \in G^n}$ and $(Y[d](k))_{k \in G^n}$, $(J[d](k))_{k \in G^n}$ is also of finite type.
		Therefore, $(J(k))_{k \in G^n}$ is of finite type.
		We have finally reduced to the case where $\myint(X(k)_x)=\emptyset$ for every $k \in G^n$ and $x \in I^{m-1}$.
		\medskip
		
		We prove $(X(k))_{k \in G^n}$ is of finite type by induction on $m$.
		Set $V:=\bigcup_{k \in G^n}X(k) \subseteq I^m$.
		We may assume that $\myint(X(k)_x)=\emptyset$ for every $k \in G^n$ and $x \in R^{m-1}$.
		By Claim 2, $V_x:=\bigcup_{k \in G^n}X(k)_x$ is a finite set for $x \in R^{m-1}$.
		
		Suppose $m=1$.
		We have $V_x=V$ in this case.
		This implies that $V$ is a finite set.
		Since $X(k)$ is a subset of $V$,  $(X(k))_{k \in G^n}$ is of finite type.
		
		Suppose $m>1$.
		By uniform finiteness property of o-minimal structures, there exists a positive integer $M$ such that the cardinality of $V_x$ is at most $M$ for every $x \in \pi(V)$.
		For $k \in G^n$ and $1 \leq i \leq M$, let $W\langle i \rangle(k)$ be the set of points $x \in I^{m-1}$ such that $X(k)_x$ contains the $i$th smallest point $v(i,x)$ of $V_x$.
		Put $$X\langle i \rangle(k) = \bigcup_{x \in W\langle i \rangle(k)} \{(x,v(i,x))\}.$$
		$(W\langle i \rangle(k))_{k \in G^n}$ is of finite type by the induction hypothesis. 
		By the definition of $X\langle i \rangle(k)$, $(X\langle i \rangle(k))_{k \in G^n}$ is also of finite type.
		Observe that $X(k)=\bigcup_{i=1}^MX\langle i \rangle(k)$ because $X(k) \subseteq V$.
		This equality implies that $(X(k))_{k \in G^n}$ is of finite type.
	\end{proof}
	
	For every $n>0$, $k\in G^n$ and definable subset $X$ of $R^n$, we set $S(X,k):=\{x \in I^n\;|\; k+x \in X\}$. 
	Observe that $(S(X,k))_{k \in G^n}$ is a uniform $\mathcal R$-definable family of subsets of $I^n$ and the equality $X=\bigcup_{k \in G^n}k+S(X,k)$ holds.
	By Claim 3, there are finitely many $\mathcal I$-definable subsets $X_1,\ldots, X_m$ of $I^n$ such that, for every $k \in G^n$, $S(X,k)=X_i$ for some $1 \leq i \leq m$.
	Put $D_i:=\{k \in G^n\;|\; S(X,k)=X_i\}$ for $1 \leq i \leq m$.
	$D_i$ is definable in $\mathcal G$.
	We have $X=\bigcup_{i=1}^m \bigcup_{k \in D_i} k+X_i$, which implies that $\mathcal R$ is interdefinable with the standard simple product of $\mathcal G$ and $\mathcal I$. 
	
	Finally, $\mathcal G$ is inp-minimal by \cite[Proposition 1.24]{Touchard} and the fact that $\mathcal R$ is interdefinable with the standard simple product of $\mathcal G$ and $\mathcal I$.
\end{proof}

We obtain a result similar to Theorem \ref{thm:local_omin_case} under a relaxed condition if $R=\mathbb R$.

\begin{proposition}\label{prop:real_case}
	Suppose that $R=\mathbb R$ and $\operatorname{Th}(\mathcal R)$ is strong.
	Then, the open core of $\mathcal R$ is either o-minimal or there exist $c>0$ and  an o-minimal structure $\mathcal I$ whose universe is $[0,c)$ such that $\mathcal R$ is interdefinable with a standard simple product of $(c\mathbb Z,<,+)$ and $\mathcal I$.
\end{proposition}
\begin{proof}
	The open core $\mathcal R^o$ of $\mathcal R$ is locally o-minimal by \cite[Theorem 1.1]{Fuji_core} and \cite[Corollary 2.13(3)]{DG}.
	If every discrete subset of $\mathbb R$ definable in $\mathcal R$ is a finite set, $\mathcal R^o$ is o-minimal.
	
	Suppose that there exists an infinite discrete subset $D$ of $\mathbb R$ definable in $\mathcal R$.
	By \cite[Proposition 2.8(1)]{FKK}, $D$ is closed because of local o-minimality. 
	We may assume that $D$ is bounded below and unbounded above without loss of generality.
	Enumerate elements of $D$ in increasing order, say $a_1, a_2, \ldots$.
	Put $d_i:=a_{i+1}-a_i$ for $i>1$.
	Let $s:D \to D$ be the map given by $s(a_i)=a_{i+1}$.
	The map $s$ is definable because it is defined by $s(x)=\inf\{u \in D\;|\;u >x\}$.
	By \cite[Proposition 2.35]{DG}, after removing a finite set from $D$, there exist $m>0$ and a sequence $\sigma:=(c_1,\ldots, c_m)$ of positive real numbers such that $d_{km+i}:=c_i$ for every $k \geq 0$ and $1 \leq i \leq m$.
	There exist $c>0$ and natural numbers $n_i$ such that $c_i=cn_i$ by \cite[Corollary 2.30]{DG}.
	We may assume that $a_1=0$ by shifting $D$ if necessarily.
	Put $D_i:=\{d \in D\;|\; s(d)-d=c_i \}$ for $1 \leq i \leq m$.
	Observe that $Z_i=\{d+kc\;|\; d \in D_i, 0 \leq k < n_i\}$ is definable.
	$c\mathbb Z$ is also definable because  $c\mathbb Z=\bigcup_{i=1}^m Z_i \cup \bigcup_{i=1}^m (-Z_i)$. 
	
	By \cite[Corollary 2.20]{DG}, $c\mathbb Z_{\mathcal R\text{-def}}$ is interdefinable with $(c\mathbb Z,<,+)$.
	This proposition now follows from \cite[Theorem 25]{KTTT}.
\end{proof}

\end{document}